\documentclass[11pt]{article}

\usepackage[margin=1in]{geometry}
\usepackage{amsmath}
\usepackage{amssymb}
\usepackage{amsthm}
\usepackage{mathtools}
\usepackage{booktabs}
\usepackage{enumitem}
\usepackage{microtype}
\usepackage[hidelinks]{hyperref}
\usepackage{xcolor} 
\usepackage{soul}   

\allowdisplaybreaks

\newtheorem{theorem}{Theorem}[section]
\newtheorem{lemma}[theorem]{Lemma}

\newtheorem{corollary}[theorem]{Corollary}

\theoremstyle{definition}
\newtheorem{definition}[theorem]{Definition}

\theoremstyle{remark}

\newcommand{\F}{\mathbb{F}}
\newcommand{\Z}{\mathbb{Z}}
\newcommand{\C}{\mathbb{C}}
\newcommand{\widehatG}{\widehat{G}}

\title{
    \textbf{
    Negative Latin-Square-Type Partial Difference Sets\\
    in Non-Elementary Abelian \(2\)-Groups
    }
}

\author{
Jos\'e S\'anchez Mench\'en, 
Dan Schwarz, Sally Brouhard,
Jacob Martin,\\
Beckett Rebele-Henry and 
Sammie Ritchie
}

\date{}

\begin{document}

\maketitle

\begin{abstract}
Using cubic cyclotomic classes, character theory, and product
constructions motivated by generalized Denniston partial difference
sets, we construct negative Latin-square-type partial difference sets
in $\F_{2^6}^{+}\times\Z_4^4$, $\Z_4^4\times\F_{2^{10}}^{+}$, $\F_{2^{10}}^{+}\times\Z_8^4$, $\Z_8^4\times\F_{2^{14}}^{+}$,
and $\Z_4^4\times\Z_{16}^2$. The first four constructions replace cubic cyclotomic partitions by
partitions of non-elementary abelian \(2\)-groups having the same
character-value patterns. To the best of our knowledge, these are the
first partial difference sets with the stated parameters in the
indicated groups.

\end{abstract}

\medskip


\section{Introduction}

Partial difference sets provide a common framework for several
combinatorial structures whose associated character sums or
intersection numbers take only two values. Examples include
projective two-weight codes, projective two-intersection sets, and
strongly regular graphs with two distinct restricted eigenvalues.

In particular, a regular partial difference set (PDS) in a finite group gives rise to a strongly regular Cayley graph. If \(D\) is a regular
PDS in a group \(G\), the corresponding Cayley
graph has vertex set \(G\), with two vertices \(x,y\in G\) adjacent
whenever $x-y\in D.$
See, for example, \cite{MaSurvey} for background on partial difference
sets and their connections with strongly regular graphs.

A large portion of the literature concentrates on the Latin-square-type
and negative Latin-square-type parameter families. Generalized
Denniston PDSs overlap with these families. These constructions are commonly obtained in elementary
abelian groups by using finite fields, cyclotomic classes, and
character theory; see \cite{denniston1969,GenDennPaper}.

Constructions in non-elementary abelian groups are less common.
Nevertheless, several authors have shown that cyclotomic constructions
can sometimes be transferred to non-elementary abelian \(p\)-groups;
see, for example,
\cite{davisxiang2000,polhill2008generalizations,polhill2009,
polhillDavisSmith2013}.

The purpose of this paper is to apply these ideas to construct new  generalized
Denniston-type constructions. More precisely, we use partitions of
\(\Z_4^4\) and \(\Z_8^4\) whose character-value patterns agree with
those of appropriate cubic cyclotomic classes \cite{private}. These partitions can
therefore be substituted for the cyclotomic classes in certain product
constructions.

The main results from this paper are the construction of four negative Latin-square-type PDSs in groups containing a non-elementary abelian \(2\)-group factor, and the creation of an additional negative Latin-square-type PDSs in \(\Z_4^4\times\Z_{16}^2\) by combining one of our partitions with a construction of Polhill, Davis, and Smith \cite{polhillDavisSmith2013}.


\section{Preliminaries}
\label{sec:preliminaries}

Throughout this paper, all groups are finite and are written
additively. We begin by recalling the definition of a partial
difference set and the parameter families relevant to our
constructions. We then introduce the character-theoretic criterion
that will be used to verify our sets. Finally, we review cubic
cyclotomic classes, which provide the character-value patterns that
motivate the constructions in the remainder of the paper.


\subsection{Partial difference sets}

\begin{definition}
A subset D of size $k$ from a finite multiplicative group G of order $v$ is called a $(v,k,\lambda,\mu)$ PDS  if the multiset of differences of D $\{ d_1-d_2 | d_1,d_2 \in D, d_1 \neq d_2 \}$, contains all non-identity elements in D $\lambda$ times and each non-identity element of G but not in D exactly $\mu$ times \cite{MaSurvey}.
\end{definition}

Among the many possible PDS parameter sets, two families appear
frequently in the literature and are especially relevant to our
constructions.

\begin{definition}
A PDS has \emph{Latin-square-type parameters} if
its parameters have the form
\[
\left(
n^2,\,
r(n-1),\,
n+r^2-3r,\,
r^2-r
\right).
\]

A PDS has
\emph{negative Latin-square-type parameters} if its parameters have
the form
\[
\left(
n^2,\,
r(n+1),\,
-n+r^2+3r,\,
r^2+r
\right).
\]
\end{definition}

The parameter sets obtained in this paper belong to the negative
Latin-square-type family. Consequently, after establishing the PDS
property, we will verify that the parameters can be written in the
second of the two forms above.


\subsection{Character theory}

Verifying that a given subset D of a finite abelian group G is PDS directly from the multiset of differences
can be difficult. For subsets of finite abelian groups, character
theory provides a more efficient method.

Let \(G\) be a finite abelian group. A \emph{character} of \(G\) is
a group homomorphism
\[
\chi:G\longrightarrow\C^\times,
\]
where \(\C^\times\) denotes the multiplicative group of nonzero
complex numbers. The collection of all characters of \(G\) forms the
character group \(\widehatG\).

The character satisfying $\chi_0(g)=1$ for every $g\in G$ is called the \emph{principal character}. All other characters are
called \emph{nonprincipal}. For a subset \(D\subseteq G\), we write
\[
\chi(D)=\sum_{d\in D}\chi(d).
\]
 The following Lemma is useful in determining whether a subset of an abelian group is a PDS.

\begin{lemma}\label{lem:character} A subset D of a finite abelian group G of order v is a $(v, k, \lambda, \mu)$ PDS if and only if
    
\[
\chi(D) = \frac{\lambda - \mu \pm \sqrt{(\lambda - \mu)^2 + 4(k-\mu)}}{2} \quad \text{for every nonprincipal character } \chi \text{ of } G
\]
\end{lemma}

This is the standard character-theoretic characterization of regular
PDSs in finite abelian groups; see
\cite{MaSurvey}. Thus, rather than counting all differences directly, it is enough to
show that every nonprincipal character sum takes one of two specified
values. The constructions in this paper are designed precisely so
that their character sums exhibit this two-valued behavior.


\subsection{Cyclotomic classes}

We next describe the principal source of the character-value patterns
used in our constructions. Although cyclotomic classes are defined
using the multiplicative structure of a finite field, we will regard
them as subsets of the additive group of that field.

Let \(q\) be a prime power, let \(\omega\) be a primitive element of
\(\F_q\), and suppose that \(e\mid(q-1)\). The \emph{cyclotomic
classes of order \(e\)} in \(\F_q\) are
\[
C_i^{(e,q)}
=
\omega^i\langle\omega^e\rangle,
\qquad
0\leq i<e.
\]

Equivalently,
\[
C_i^{(e,q)}
=
\left\{
\omega^{i+ej}:
0\leq j<\frac{q-1}{e}
\right\}.
\]

These classes partition the nonzero field elements:
\[
\F_q^\times
=
C_0^{(e,q)}
\cup
C_1^{(e,q)}
\cup\cdots\cup
C_{e-1}^{(e,q)}.
\]

In this paper, we use cyclotomic classes of order \(3\), which are
called \emph{cubic cyclotomic classes}. Since $2^{2n}\equiv 1\pmod 3$, the quantity \(2^{2n}-1\) is divisible by \(3\), so cubic
cyclotomic classes exist in every field \(\F_{2^{2n}}\).

The following result describes when these classes themselves form
PDSs in the additive group of the field; see \cite{BrouwerCyclotomy}. .

\begin{theorem}
\label{thm:cyclotomy}
Let \(n\) be a positive integer. The three cubic cyclotomic classes $C_i^{(3,2^{2n})}$, $0\leq i\leq 2$, partition the nonzero elements of \(\F_{2^{2n}}\), and each class has
cardinality $\frac{2^{2n}-1}{3}.$

If \(n\) is even, then each \(C_i^{(3,2^{2n})}\), viewed as a subset of the additive
group \(\F_{2^{2n}}^+\), is a negative Latin-square-type PDS with
parameters
\[
\left(
2^{2n},
\frac{2^{2n}-1}{3},
-2^n+(\frac{2^n-1}{3})^2+3(\frac{2^n-1}{3}),
(\frac{2^n-1}{3})^2+(\frac{2^n-1}{3})
\right).
\]

If \(n\) is odd, then each \(C_i^{(3,2^{2n})}\), viewed as a subset of
\(\F_{2^{2n}}^+\), is a Latin-square-type PDS with parameters
\[
\left(
2^{2n},
\frac{2^{2n}-1}{3},
2^n+(\frac{2^n+1}{3})^2-3(\frac{2^n+1}{3}),
(\frac{2^n+1}{3})^2-(\frac{2^n+1}{3})
\right).
\]
\end{theorem}

For our later product constructions, the most important consequence
of Theorem~\ref{thm:cyclotomy} is not only that these cyclotomic
classes are PDSs, but also that their character sums follow a highly
structured pattern. We record the particular cases needed below.

\begin{lemma}
\label{lem:cyclotomic-character-patterns}
For every nonprincipal additive character of the relevant field, the
following character-value patterns hold, up to a permutation of the
three indices, $\chi$:
\[
\left(
\chi(C_0^{(3,64)}),
\chi(C_1^{(3,64)}),
\chi(C_2^{(3,64)})
\right)
=
(5,-3,-3),
\]
\[
\left(
\chi(C_0^{(3,1024)}),
\chi(C_1^{(3,1024)}),
\chi(C_2^{(3,1024)})
\right)
=
(21,-11,-11),
\]
and
\[
\left(
\chi(C_0^{(3,16384)}),
\chi(C_1^{(3,16384)}),
\chi(C_2^{(3,16384)})
\right)
=
(85,-43,-43).
\]
\end{lemma}

\begin{proof}
Applying Lemma~\ref{lem:character} to the parameter sets in
Theorem~\ref{thm:cyclotomy} gives the two possible character values
for each cyclotomic class.

Moreover, the three cubic cyclotomic classes partition the nonzero
elements of the relevant field. Therefore, for every nonprincipal
additive character \(\chi\),
\[
\sum_{i=0}^{2}\chi\left(C_i^{(3,q)}\right)
=
\chi(\F_q^\times)
=
-1.
\]

In each of the three cases, the only way for three allowed character
values to have sum \(-1\) is for the larger value to occur once and
the smaller value to occur twice. This gives the stated patterns.
\end{proof}

The significance of these patterns is that our partitions of
\(\Z_4^4\) and \(\Z_8^4\) will produce the same distributions of
character values. This allows those partitions to replace cubic
cyclotomic classes in the product constructions developed later.


\subsection{The generalized Denniston framework}

Generalized Denniston PDSs are constructed by
combining compatible cyclotomic partitions in additive groups of
finite fields. The construction described by Li, Davis, Huczynska,
Johnson, and Polhill \cite{GenDennPaper} depends on the character-value
patterns of the relevant partitions.

The essential observation for the present paper is that the same
character-sum calculations remain valid when one cyclotomic
partition is replaced by a partition in another abelian group having
the same character-value pattern. Our constructions use this
replacement principle with partitions of \(\Z_4^4\) and \(\Z_8^4\).


\section{Partitions in Non-Elementary Abelian Groups}
\label{sec:partitions}

We use the following partitions.

\begin{theorem}
\label{thm:partitions}
There exist subsets $E_0,E_1,E_2\subseteq\Z_4^4$
that partition $\Z_4^4\setminus\{0\},$
such that each \(E_i\) is a $(256,85,24,30)\text{-PDS}.$

There also exist subsets $H_0,H_1,H_2\subseteq\Z_8^4$
that partition $\Z_8^4\setminus\{0\},$
such that each \(H_i\) is a $(4096,1365,440,462)\text{-PDS}.$
\end{theorem}

%
%

The first partition is explicitly listed in Appendix~\ref{app:Esets}, following the construction in \cite{private}. A direct computation verifies that the sets $(E_0,E_1,E_2)$ partition the nonzero elements of $\Z_4^4$ and satisfy the stated parameters. The partition of $\Z_8^4\setminus{\{0\}}$ into $(H_0,H_1,H_2)$ is omitted from the appendix due to the large size of the sets.
 \\

The character patterns of these partitions are determined by their
PDS parameters.

\begin{lemma}
\label{lem:partition-character-patterns}
Let \(\chi\) be a nonprincipal character of \(\Z_4^4\). Then there
is a unique index \(a\in\{0,1,2\}\) such that $\chi(E_a)=-11,$ while $\chi(E_i)=5$ for $i\neq a$.

Let \(\chi\) be a nonprincipal character of \(\Z_8^4\). Then there
is a unique index \(b\in\{0,1,2\}\) such that
\[
\chi(H_b)=-43,
\]
while
\[
\chi(H_i)=21
\qquad
\text{for }i\neq b.
\]
\end{lemma}

\begin{proof}
For a \((256,85,24,30)\)-PDS, Lemma~\ref{lem:character}
gives
\[
\chi(E_i)\in
\left\{
\frac{-6+\sqrt{(-6)^2+4(85-30)}}{2},
\frac{-6-\sqrt{(-6)^2+4(85-30)}}{2}
\right\}.
\]
Since $(-6)^2+4(55)=256,$ the possible values are $5$ and $-11$. Because \(E_0,E_1,E_2\) partition the nonzero elements of \(\Z_4^4\), $\chi(E_0)+\chi(E_1)+\chi(E_2)=-1.$
The only way three numbers, each equal to \(5\) or \(-11\), can sum to \(-1\) is for one value to be \(-11\) and the other two to be
\(5\).

Similarly, for a \((4096,1365,440,462)\)-PDS, the two character values are $21$ and $-43$.

Since
\[
\chi(H_0)+\chi(H_1)+\chi(H_2)=-1,
\]
exactly one value is \(-43\) and the other two are \(21\).
\end{proof}

\begin{corollary}
\label{cor:adjoined-zero-patterns}
For a nonprincipal character of \(\Z_4^4\), the character values on $E_i\cup\{0\}$
have pattern $(-10,6,6),$ up to a permutation of the indices. 
\end{corollary}

\begin{corollary}
    For a nonprincipal character of \(\Z_8^4\), the character values on $H_i\cup\{0\}$
have pattern $(-42,22,22),$ up to a permutation of the indices.
\end{corollary}

\begin{proof}
Every character takes the value \(1\) at the identity element.
Therefore,
\[
\chi(E_i\cup\{0\})=\chi(E_i)+1
\]
and
\[
\chi(H_i\cup\{0\})=\chi(H_i)+1.
\]
The result now follows from Lemma
\ref{lem:partition-character-patterns}.
\end{proof}


\section{Main Result}
\label{sec:mainproof}

\begin{theorem}
\label{thm:main}
There exist negative Latin-square-type PDSs with
the following parameters and ambient groups:

\begin{enumerate}[label=(\arabic*)]
    \item a $(16384,5418,1762,1806)$-PDS in
   $\F_{2^6}^+\times\Z_4^4;$

    \item a $(262144,87210,28898,29070)\text{-PDS}
    $
    in
    $\Z_4^4\times\F_{2^{10}}^+;$

    \item a $(4194304,1397418,465122,465806)\text{-PDS}$ in
    $\F_{2^{10}}^+\times\Z_8^4;$

    \item a
    $(67108864,22366890,7452898,7455630)\text{-PDS}$
    in
    $\Z_8^4\times\F_{2^{14}}^+.$
\end{enumerate}
\end{theorem}

It is worth noticing that the parameter sets in Theorem~\ref{thm:main} are simultaneously
negative Latin-square-type parameters and generalized Denniston
parameters.

\begin{proof}
Define
\[
D_1
=
\bigcup_{i=0}^{2}
\left(
C_i^{(3,64)}
\times
(E_i\cup\{0\})
\right)
\subseteq
\F_{2^6}^+\times\Z_4^4,
\]
\[
D_2
=
\bigcup_{i=0}^{2}
\left(
E_i
\times
(C_i^{(3,1024)}\cup\{0\})
\right)
\subseteq
\Z_4^4\times\F_{2^{10}}^+,
\]
\[
D_3
=
\bigcup_{i=0}^{2}
\left(
C_i^{(3,1024)}
\times
(H_i\cup\{0\})
\right)
\subseteq
\F_{2^{10}}^+\times\Z_8^4,
\]
and
\[
D_4
=
\bigcup_{i=0}^{2}
\left(
H_i
\times
(C_i^{(3,16384)}\cup\{0\})
\right)
\subseteq
\Z_8^4\times\F_{2^{14}}^+.
\]

The identity is not contained in
any \(D_j\), and each constituent set is closed under additive
inverses.

The cardinalities are
\[
|D_1|
=
3\left(\frac{64-1}{3}\right)(85+1)
=
3(21)(86)
=
5418,
\]
\[
|D_2|
=
3(85)
\left(\frac{1024-1}{3}+1\right)
=
3(85)(342)
=
87210,
\]
\[
|D_3|
=
3\left(\frac{1024-1}{3}\right)(1365+1)
=
3(341)(1366)
=
1397418,
\]
\[
|D_4|
=
3(1365)
\left(\frac{16384-1}{3}+1\right)
=
3(1365)(5462)
=
22366890.
\]

Every character of a direct product is a product character. Thus,
for example, every character of
\(\F_{2^6}^+\times\Z_4^4\)
has the form
\[
\psi\times\varphi,
\]
where \(\psi\) is a character of \(\F_{2^6}^+\) and \(\varphi\)
is a character of \(\Z_4^4\).

The relevant character-value patterns are summarized below. Each
three-term pattern is understood up to a simultaneous permutation of
the indices.

\[
\begin{array}{c|c|c|c|c}
\toprule
\text{Set}
&
\text{First-factor pattern}
&
\text{Second-factor pattern}
&
\text{First size}
&
\text{Second size}
\\
\midrule
D_1
&
(5,-3,-3)
&
(-10,6,6)
&
21
&
86
\\
D_2
&
(-11,5,5)
&
(22,-10,-10)
&
85
&
342
\\
D_3
&
(21,-11,-11)
&
(-42,22,22)
&
341
&
1366
\\
D_4
&
(-43,21,21)
&
(86,-42,-42)
&
1365
&
5462
\\
\bottomrule
\end{array}
\]

We show the complete computation for \(D_1\).

If both \(\psi\) and \(\varphi\) are principal, then
\[
(\psi\times\varphi)(D_1)
=
3(21)(86)
=
5418.
\]

Suppose that \(\psi\) is nonprincipal and \(\varphi\) is principal.
Then
\[
(\psi\times\varphi)(D_1)
=
86(5-3-3)
=
-86.
\]

Suppose that \(\psi\) is principal and \(\varphi\) is nonprincipal.
Then
\[
(\psi\times\varphi)(D_1)
=
21(-10+6+6)
=
42.
\]

Finally, suppose both characters are nonprincipal. If the exceptional
indices of the two character patterns coincide, then
\[
(\psi\times\varphi)(D_1)
=
5(-10)+(-3)6+(-3)6
=
-86.
\]
If the exceptional indices are distinct, then
\[
(\psi\times\varphi)(D_1)
=
5(6)+(-3)(-10)+(-3)6
=
42.
\]

For the proposed parameters of \(D_1\),
\[
\lambda-\mu
=
1762-1806
=
-44
\]
and
\[
(\lambda-\mu)^2+4(k-\mu)
=
(-44)^2+4(5418-1806)
=
16384
=
128^2.
\]
The two nonprincipal character values required by
Lemma~\ref{lem:character} are therefore
\[
\frac{-44+128}{2}=42
\]
and
\[
\frac{-44-128}{2}=-86.
\]
Thus, \(D_1\) is a $(16384,5418,1762,1806)\text{-PDS}.$ Similar calculations demonstrate that  \(D_2,D_3,D_4\) are PDSs. The resulting
character sums are:

\[
\begin{array}{c|c|c|c}
\toprule
\text{Set}
&
\text{Principal value}
&
\multicolumn{2}{c}{\text{Nonprincipal values}}
\\
\midrule
D_1 & 5418     & -86   & 42   \\
D_2 & 87210    & -342  & 170  \\
D_3 & 1397418  & -1366 & 682  \\
D_4 & 22366890 & -5462 & 2730 \\
\bottomrule
\end{array}
\]

For \(D_2\), the required roots are
\[
\frac{28898-29070
\pm
\sqrt{(28898-29070)^2+4(87210-29070)}}{2}
=
170,-342.
\]

For \(D_3\), the required roots are
\[
\frac{465122-465806
\pm
\sqrt{(465122-465806)^2+
4(1397418-465806)}}{2}
=
682,-1366.
\]

For \(D_4\), the required roots are
\[
\frac{7452898-7455630
\pm
\sqrt{(7452898-7455630)^2+
4(22366890-7455630)}}{2}
=
2730,-5462.
\]

Therefore, Lemma~\ref{lem:character} proves that all four sets have
the stated PDS parameters.

It remains to verify that the parameter sets are of negative
Latin-square type. They correspond respectively to
\[
(n,r)=(128,42),\quad
(512,170),\quad
(2048,682),\quad
(8192,2730).
\]
In each case,
\[
v=n^2,\qquad
k=r(n+1),\qquad
\lambda=-n+r^2+3r,\qquad
\mu=r^2+r.
\]
Hence all four PDSs are of negative Latin-square type.
\end{proof}


\section{An Additional Construction}
\label{sec:other}

Polhill, Davis, and Smith constructed a PDS in $\Z_2^8\times\Z_{16}^2$
by combining cubic cyclotomic classes in \(\Z_2^8\) with a partition
of \(\Z_{16}^2\); see Example~4.3 of
\cite{polhillDavisSmith2013}. Their construction uses subsets $J_0,J_1,J_2\subseteq\Z_{16}^2$
that partition the nonzero elements of \(\Z_{16}^2\), with $|J_0|=90, |J_1|=90, |J_2|=75.$

We replace the cubic cyclotomic partition in the first factor with the partition $E_0,E_1,E_2$
of \(\Z_4^4\). This construction helps us reach the following theorem. 

\begin{theorem}
\label{thm:secondMain}
There exists a $(65536,21845,7224,7310)$
negative Latin-square-type PDS in
$\Z_4^4\times\Z_{16}^2.$
\end{theorem}

\begin{proof}
Define
\[
\begin{aligned}
D_5
={}&
\left(E_0\times(J_0\cup\{0\})\right)
\\
&\cup
\left(E_1\times(J_1\cup\{0\})\right)
\\
&\cup
\left(E_2\times J_2\right).
\end{aligned}
\]

Its cardinality is
\[
\begin{aligned}
|D_5|
&=
85(91)+85(91)+85(75)
\\
&=
85(257)
\\
&=
21845.
\end{aligned}
\]

The character values of the partition \(E_0,E_1,E_2\) agree with
the corresponding character values required in the product
construction of \cite{polhillDavisSmith2013}. Substituting the
\(E_i\) for the original first-factor partition therefore gives the
same two nonprincipal character values as in that construction.

For the proposed parameters,
\[
\lambda-\mu
=
7224-7310
=
-86,
\]
and
\[
\begin{aligned}
(\lambda-\mu)^2+4(k-\mu)
&=
(-86)^2+4(21845-7310)
\\
&=
65536
\\
&=
256^2.
\end{aligned}
\]
Thus, the two required nonprincipal character values are $\frac{-86+256}{2}=85$ and $\frac{-86-256}{2}=-171.$ Consequently, \(D_5\) is a $(65536,21845,7224,7310)\text{-PDS}.$\\

Finally, taking $n=256$ and $r=85$, we obtain $v=n^2=65536,$ $k=r(n+1)=85(257)=21845,$
$\lambda=-n+r^2+3r=7224,$ and $\mu=r^2+r=7310.$
Therefore, the parameters are of negative Latin-square type.
\end{proof}



\section{Future Directions}

The constructions in this paper show that certain cyclotomic
partitions in elementary abelian groups can be replaced by partitions
in non-elementary abelian groups having the same character-value
patterns. The main ingredient is the existence of partitions of groups such as $\Z_{2^a}^4$ into three PDSs whose character sums agree with
those of appropriate cubic cyclotomic classes. This motivates the following questions.

\begin{enumerate}
    \item Are there other non-elementary abelian \(p\)-groups whose
    nonzero elements can be partitioned into three equal-size partial
    difference sets with the character-value patterns of cubic
    cyclotomic classes?

    \item Can similar replacement techniques be applied to
    cyclotomic constructions of order greater than three?

    \item Are there nonabelian groups containing partial difference
    sets with the same parameters as the constructions in this
    paper?

    \item Are there nonabelian groups whose nonidentity elements can
    be partitioned into three PDSs with compatible
    character-theoretic or representation-theoretic behavior?

    \item Can the constructions presented here be described through
    a uniform algebraic rule rather than explicit lists or computer
    searches?
\end{enumerate}

Answers to these questions may produce broader families of partial
difference sets in non-elementary abelian and nonabelian groups.


\section*{Acknowledgments}

The authors would like to thank Jim Davis for his guidance and mentorship throughout this project, as well as John Polhill and Eric Swartz for their help and suggestions. Moreover, we would also like to thank Jonathan Jedwab and his team for sharing unpublished results that were used in the constructions in this paper.

%


\appendix

\section{The Partition of \(\Z_4^4\)}
\label{app:Esets}

The following sets partition
\[
\Z_4^4\setminus\{(0,0,0,0)\}.
\]

\begingroup
\scriptsize

\[
E_0=
\left\{
\begin{aligned}
&(2,0,0,0),(0,0,0,2),(0,0,2,2),(0,2,0,2),(2,2,2,2),
(1,0,0,0),(3,0,0,0),(0,0,0,1),(0,0,0,3),\\
&(2,2,3,1),(2,2,1,3),(0,1,2,1),(0,3,2,3),(1,1,3,1),
(3,3,1,3),(1,2,0,0),(3,2,0,0),(2,2,0,1),\\
&(2,2,0,3),(0,0,3,3),(0,0,1,1),(2,3,0,1),(2,1,0,3),
(3,1,1,3),(1,3,3,1),(1,0,0,2),(3,0,0,2),\\
&(0,0,2,3),(0,0,2,1),(2,0,3,3),(2,0,1,1),(2,3,0,3),
(2,1,0,1),(3,1,3,1),(1,3,1,3),(1,0,2,2),\\
&(3,0,2,2),(0,2,0,3),(0,2,0,1),(0,0,1,3),(0,0,3,1),
(2,1,2,1),(2,3,2,3),(1,1,3,3),(3,3,1,1),\\
&(0,1,0,2),(0,3,0,2),(3,1,0,2),(1,3,0,2),(3,1,0,1),
(1,3,0,3),(1,3,1,0),(3,1,3,0),(1,2,3,3),\\
&(3,2,1,1),(0,1,2,2),(0,3,2,2),(3,3,2,2),(1,1,2,2),
(1,1,2,1),(3,3,2,3),(1,1,3,2),(3,3,1,2),\\
&(3,2,3,1),(1,2,1,3),(1,2,3,0),(2,1,1,1),(2,3,3,3),
(1,2,2,1),(3,2,2,3),(0,2,1,2),(0,2,3,2),\\
&(2,1,3,0),(2,3,1,0),(3,2,1,0),(2,0,1,2),(2,0,3,2),
(0,3,3,2),(0,1,1,2),(1,0,3,0),(3,0,1,0),\\
&(0,1,3,3),(0,3,1,1),(1,0,2,1),(3,0,2,3)
\end{aligned}
\right\}.
\]

\[
E_1=
\left\{
\begin{aligned}
&(0,2,0,0),(2,2,0,0),(2,2,0,2),(2,2,2,0),(2,0,2,2),
(1,2,0,2),(3,2,0,2),(2,2,2,3),(2,2,2,1),\\
&(0,2,3,1),(0,2,1,3),(0,1,2,3),(0,3,2,1),(1,1,1,3),
(3,3,3,1),(1,2,2,2),(3,2,2,2),(2,0,0,3),\\
&(2,0,0,1),(2,2,1,1),(2,2,3,3),(0,3,0,1),(0,1,0,3),
(3,1,1,1),(1,3,3,3),(0,1,0,0),(0,3,0,0),\\
&(3,1,2,0),(1,3,2,0),(3,3,0,3),(1,1,0,1),(3,1,3,2),
(1,3,1,2),(3,2,3,3),(1,2,1,1),(0,1,2,0),\\
&(0,3,2,0),(3,3,0,0),(1,1,0,0),(1,3,2,3),(3,1,2,1),
(3,3,1,0),(1,1,3,0),(1,2,3,1),(3,2,1,3),\\
&(2,1,0,0),(2,3,0,0),(3,1,2,2),(1,3,2,2),(3,3,2,1),
(1,1,2,3),(3,3,3,0),(1,1,1,0),(1,0,1,1),\\
&(3,0,3,3),(2,1,0,2),(2,3,0,2),(3,1,0,0),(1,3,0,0),
(3,1,2,3),(1,3,2,1),(1,1,1,2),(3,3,3,2),\\
&(3,0,1,1),(1,0,3,3),(2,2,1,0),(2,2,3,0),(2,1,1,0),
(2,3,3,0),(3,0,3,0),(1,0,1,0),(0,1,3,1),\\
&(0,3,1,3),(1,0,0,3),(3,0,0,1),(2,2,1,2),(2,2,3,2),
(2,1,3,2),(2,3,1,2),(3,2,3,2),(1,2,1,2),\\
&(2,3,1,3),(2,1,3,1),(3,0,0,3),(1,0,0,1)
\end{aligned}
\right\}.
\]

\[
E_2=
\left\{
\begin{aligned}
&(0,0,2,0),(0,2,2,0),(2,0,2,0),(0,2,2,2),(2,0,0,2),
(1,0,2,0),(3,0,2,0),(0,2,2,1),(0,2,2,3),\\
&(0,2,1,1),(0,2,3,3),(0,3,0,3),(0,1,0,1),(3,1,3,3),
(1,3,1,1),(1,2,2,0),(3,2,2,0),(2,0,2,1),\\
&(2,0,2,3),(2,0,1,3),(2,0,3,1),(2,1,2,3),(2,3,2,1),
(1,1,1,1),(3,3,3,3),(2,1,2,0),(2,3,2,0),\\
&(3,3,0,2),(1,1,0,2),(1,3,0,1),(3,1,0,3),(3,1,1,2),
(1,3,3,2),(3,0,1,3),(1,0,3,1),(2,1,2,2),\\
&(2,3,2,2),(3,3,2,0),(1,1,2,0),(1,1,0,3),(3,3,0,1),
(1,3,3,0),(3,1,1,0),(1,0,1,3),(3,0,3,1),\\
&(0,0,1,0),(0,0,3,0),(0,3,1,2),(0,1,3,2),(1,2,1,0),
(3,2,3,0),(2,1,1,3),(2,3,3,1),(1,2,0,3),\\
&(3,2,0,1),(0,0,1,2),(0,0,3,2),(0,3,3,0),(0,1,1,0),
(1,0,1,2),(3,0,3,2),(0,3,3,1),(0,1,1,3),\\
&(3,2,0,3),(1,2,0,1),(0,2,1,0),(0,2,3,0),(2,1,1,2),
(2,3,3,2),(3,0,1,2),(1,0,3,2),(0,3,3,3),\\
&(0,1,1,1),(3,2,2,1),(1,2,2,3),(2,0,1,0),(2,0,3,0),
(0,3,1,0),(0,1,3,0),(1,2,3,2),(3,2,1,2),\\
&(2,3,1,1),(2,1,3,3),(3,0,2,1),(1,0,2,3)
\end{aligned}
\right\}.
\]

\endgroup


\end{document}